\documentclass[10pt]{amsart}

\usepackage{amssymb}
\usepackage{amsthm}
\usepackage{amsmath}

\usepackage[usenames]{color}
\definecolor{ANDREW}{RGB}{255,127,0}

\usepackage{tikz}
\usetikzlibrary{arrows,calc}

\usepackage[foot]{amsaddr}

\usepackage{hyperref}

\theoremstyle{plain}
\newtheorem{proposition}{Proposition}[section]
\newtheorem{theorem}[proposition]{Theorem}
\newtheorem{lemma}[proposition]{Lemma}
\newtheorem{corollary}[proposition]{Corollary}
\theoremstyle{definition}
\newtheorem{example}[proposition]{Example}
\newtheorem{definition}[proposition]{Definition}
\newtheorem{observation}[proposition]{Observation}
\theoremstyle{remark}
\newtheorem{remark}[proposition]{Remark}

\newtheorem*{question}{Question}

\DeclareMathOperator{\Aff}{Aff}

\DeclareMathOperator{\dimension}{dim}

\DeclareMathOperator{\Real}{Re}
\DeclareMathOperator{\Imaginary}{Im}

\DeclareMathOperator{\GL}{GL}

\DeclareMathOperator{\ord}{ord} 
 
\DeclareMathOperator{\Span}{Span} 
 
\DeclareMathOperator{\Euc}{Euc}

\DeclareMathOperator{\Bc}{\mathcal{B}}
\DeclareMathOperator{\Cc}{\mathcal{C}}

\DeclareMathOperator{\Gc}{\mathcal{G}}

\DeclareMathOperator{\Nc}{\mathcal{N}}
\DeclareMathOperator{\Oc}{\mathcal{O}}

\DeclareMathOperator{\Uc}{\mathcal{U}}

\DeclareMathOperator{\Cb}{\mathbb{C}}
\DeclareMathOperator{\Db}{\mathbb{D}}
\DeclareMathOperator{\Fb}{\mathbb{F}}

\DeclareMathOperator{\Kb}{\mathbb{K}}
\DeclareMathOperator{\Lb}{\mathbb{L}}
\DeclareMathOperator{\Nb}{\mathbb{N}}
\DeclareMathOperator{\Pb}{\mathbb{P}}
\DeclareMathOperator{\Rb}{\mathbb{R}}
\DeclareMathOperator{\Xb}{\mathbb{X}}

\newcommand{\abs}[1]{\left|#1\right|}

\newcommand{\norm}[1]{\left\|#1\right\|}

\newcommand{\wh}[1]{\widehat{#1}}
\newcommand{\ip}[1]{\left\langle #1\right\rangle}

\begin{document}

\title{A gap theorem for the complex geometry of convex domains}
\author{Andrew Zimmer}\address{Department of Mathematics, University of Chicago, Chicago, IL 60637.}
\email{aazimmer@uchicago.edu}
\date{\today}
\subjclass[2010]{}

\begin{abstract} In this paper we establish a gap theorem for the complex geometry of smoothly bounded convex domains which informally says that if the complex geometry near the boundary is close to the complex geometry of the unit ball, then the domain must be strongly pseudoconvex. 

One consequence of our general result is the following:  for any dimension there exists some $\epsilon > 0$ so that if the squeezing function on a smoothly bounded convex domain is greater than $1-\epsilon$ outside a compact set, then the domain is strongly pseudoconvex (and hence the squeezing function limits to one on the boundary). Another consequence is the following:  for any dimension $d$  there exists some $\epsilon > 0$ so that if the holomorphic sectional curvature of the Bergman metric on a smoothly bounded convex domain is within $\epsilon$ of $-4/(d+1)$ outside a compact set, then the domain is strongly pseudoconvex (and hence the holomorphic sectional curvature limits to $-4/(d+1)$ on the boundary). 
\end{abstract}

\maketitle

\section{Introduction}

There are many results showing that the asymptotic complex geometry of a strongly pseudoconvex domain coincides with the complex geometry of  the unit ball. In this paper we consider the following related question:

\begin{question}
Suppose $\Omega \subset \Cb^d$ is a bounded pseudoconvex domain with $C^\infty$ boundary. If the asymptotic complex geometry of $\Omega$ coincides with the complex geometry of the unit ball, is $\Omega$ strongly pseudoconvex? 
\end{question}

We will restrict our attention to convex domains and for such domains answer the above question in the affirmative (see Theorems~\ref{thm:cont} and~\ref{thm:upper_cont} below). We begin by stating some consequences of our general results. 

\subsection{The squeezing function}

Given a domain $\Omega \subset \Cb^d$ let $s_\Omega : \Omega \rightarrow (0,1]$ be the \emph{squeezing function on $\Omega$}, that is 
\begin{align*}
s_\Omega(p) = \sup\{ r : & \text{ there exists an one-to-one holomorphic map } \\
& f: \Omega \rightarrow \Bc \text{ with } f(p)=0 \text{ and } r\Bc \subset f(\Omega) \}.
\end{align*}
The squeezing function has a number of applications (see for instance~\cite{LSY2004,Y2009}). 

Theorems of Diederich, Forn{\ae}ss, and Wold~\cite[Theorem 1.1]{DFW2014} and Deng, Guan, and Zhang~\cite[Theorem 1.1]{DGZ2016} imply the following asymptotic result for the squeezing function on strongly pseudoconvex domains:

\begin{theorem}~\cite{DFW2014, DGZ2016}
Suppose $\Omega \subset \Cb^d$ is a bounded strongly pseudoconvex domain. Then 
\begin{align*}
\lim_{z \rightarrow \partial \Omega} s_\Omega(z) = 1.
\end{align*}
\end{theorem}

 Based on the above theorem, it is natural to ask if the converse holds: 

\begin{question} (Forn{\ae}ss) Suppose $\Omega  \subset \Cb^d$ is a bounded pseudoconvex domain with $C^\infty$ boundary. If 
\begin{align*}
\lim_{z \rightarrow \partial \Omega} s_\Omega(z) = 1,
\end{align*}
is $\Omega$ strongly pseudoconvex?
\end{question}

In this paper we answer the question for convex domains and in this case prove a stronger assertion: 

\begin{theorem}\label{thm:squeezing}(see Section~\ref{sec:squeezing})
For any $d > 0$, there exists some $\epsilon=\epsilon(d) > 0$ so that: if $\Omega  \subset \Cb^d$ is a bounded convex domain with $C^\infty$ boundary and 
\begin{align*}
s_\Omega(z) \geq 1-\epsilon
\end{align*}
outside a compact subset of $\Omega$, then $\Omega$ is strongly pseudoconvex. 
\end{theorem}

\subsection{Holomorphic sectional curvature of the Bergman metric}

Let $(X,J)$ be a complex manifold with K{\"a}hler metric $g$. If $R$ is the Riemannian curvature tensor of $(X,g)$, then the \emph{holomorphic sectional curvature} $H_g(v)$ of a nonzero vector $v$ is defined to be the sectional curvature of the 2-plane spanned by $v$ and $Jv$, that is 
\begin{align*}
H_g(v) := \frac{ R(v,Jv,Jv,v)}{\norm{v}_g^4}.
\end{align*}
It is a classical result of Hawley~\cite{H1953} and Igusa~\cite{I1954} that if $(X,g)$ is a complete simply connected K{\"a}hler manifold with constant negative holomorphic sectional curvature, then $X$ is biholomorphic to the unit ball (also see Chapter IX, Section 7 in~\cite{KN1996}). Moreover, if $b_{\Bc}$ is the Bergman metric on the unit ball $\Bc \subset \Cb^d$, then $(\Bc, b_{\Bc})$ has constant holomorphic sectional curvature $-4/(d+1)$. 

Klembeck proved the following asymptotic result for the Bergman metric on a strongly pseudoconvex domain:

\begin{theorem}[Klembeck~\cite{K1978}]
Suppose $\Omega \subset \Cb^d$ is a bounded strongly pseudoconvex domain. Then 
\begin{align*}
\lim_{ z \rightarrow \partial \Omega} \max_{v \in T_z \Omega \setminus \{0\}} \abs{H_{b_\Omega}(v) - \frac{-4}{d+1}} = 0
\end{align*}
where $b_\Omega$ is the Bergman metric on $\Omega$. 
\end{theorem}

In this paper we will prove the following converse to Klembeck's theorem:

\begin{theorem}\label{thm:bergman}(see Section~\ref{sec:bergman})
For any $d > 0$, there exists some $\epsilon=\epsilon(d) > 0$ so that: if $\Omega  \subset \Cb^d$ is a bounded convex domain with $C^\infty$ boundary and 
\begin{align*}
\max_{v \in T_z \Omega \setminus \{0\}} \abs{H_{b_\Omega}(v) - \frac{-4}{d+1}} \leq \epsilon
\end{align*}
outside a compact subset of $\Omega$, then $\Omega$ is strongly pseudoconvex. 
\end{theorem}

\subsection{K{\"a}hler metrics with bounded geometry}

Theorem~\ref{thm:bergman} actually holds for a much larger class of K{\"a}hler metrics with bounded geometry. Before stating our result we will need to rigorously define what we mean by ``K{\"a}hler metrics with bounded geometry.''

\begin{definition}
Suppose $\Omega \subset \Cb^d$ is a bounded domain and $M > 1$. Let $\Gc_M(\Omega)$ be the set of K{\"a}her metrics $g$ on $\Omega$ (with respect to the standard complex structure)
 with the following properties:
\begin{enumerate}
\item $g$ is a $C^2$ metric, 
\item For all $z \in \Omega$ and $v \in \Cb^d$,  
\begin{align*}
\frac{1}{M} g_z(v,v) \leq k_\Omega(z;v) \leq M g_z(v,v)
\end{align*}
where $k_\Omega$ is the infinitesimal Kobayashi metric on $\Omega$,
\item If $X, v, w \in \Cb^d$ then 
\begin{align*}
\abs{X(g_z(v,w))} \leq M k_\Omega(z;X)k_\Omega(z;v)k_\Omega(z;w),
\end{align*}
\item If $X,Y, v, w \in \Cb^d$ then 
\begin{align*}
\abs{Y(X(g_z(v,w)))} \leq M k_\Omega(z;Y)k_\Omega(z;X)k_\Omega(z;v)k_\Omega(z;w),
\end{align*} 
\item If $X,Y, v, w \in \Cb^d$ and $z_1,z_2 \in \Omega$, then 
\begin{align*}
\abs{Y(X(g_{z_1}(v,w)))-Y(X(g_{z_2}(v,w)))} \leq Mk_\Omega(z;Y)k_\Omega(z;X)k_\Omega(z;v)k_\Omega(z;w) K_\Omega(z_1,z_2).
\end{align*} 
where $K_\Omega$ is the Kobayashi distance on $\Omega$.
\end{enumerate}
\end{definition}

\begin{remark} \ \begin{enumerate}
\item The definition above essentially says that $\Gc_M(\Omega)$ is the set of all K{\"a}hler metrics  whose second derivative is bounded and Lipschitz with respect to the Kobayashi metric. 
\item Here we use the standard notation: if $f:\Omega \rightarrow \Cb$ is a $C^1$ function and $X \in \Cb^d$ is some vector, then we define the function $X(f): \Omega \rightarrow \Cb$ by 
\begin{align*}
X(f)(z) =\left. \frac{d}{dt} \right|_{t=0} f(z+tX).
\end{align*}
\item We will prove that for any $d > 0$ there exists an $M = M(d) > 0$ so that: if $\Omega \subset \Cb^d$ is a bounded convex domain in $\Cb^d$, then $b_\Omega \in \Gc_M(\Omega)$. 
\end{enumerate}
\end{remark}

Theorem~\ref{thm:bergman} will be a consequence of the following more general result: 

\begin{theorem}\label{thm:gen_riem}(see Section~\ref{sec:gen_riem})
For any $d > 0$ and $M > 1$, there exists some $\epsilon = \epsilon(M,d) > 0$ so that: if $\Omega \subset \Cb^d$ is a bounded convex domain with $C^\infty$ boundary and there exists a metric $g \in \Gc_M(\Omega)$ with
\begin{align*}
\max_{v,w \in T_z \Omega \setminus \{0\}} \abs{H_{g}(v) - H_g(w)} \leq \epsilon
\end{align*}
outside a compact subset of $\Omega$, then $\Omega$ is strongly pseudoconvex.
\end{theorem}

\subsection{The general theorem}

Theorems~\ref{thm:squeezing},~\ref{thm:bergman}, and~\ref{thm:gen_riem} are particular cases of a more general gap theorem which we now describe. In order to state our main result we need to define the space of convex domains and intrinsic functions on them.

\begin{definition} Let $\Xb_d$ be the set of convex domains in $\Cb^d$ which do not contain a complex affine line and let $\Xb_{d,0}$ be the set of pairs $(\Omega,x)$ where $\Omega \in \Xb_d$ and $x \in \Omega$. \end{definition}

\begin{remark}\label{rmk:barth} When $\Omega$ is a convex domain, Barth~\cite{B1980} proved that the following are equivalent: 
\begin{enumerate}
\item $\Omega$ contains no complex affine lines,
\item the Kobayashi metric is non-degenerate, 
\item the Kobayashi metric is Cauchy complete.
\end{enumerate}
Thus, from a complex geometric point of view, it is natural to study the convex domains which do not contain any affine lines. 
\end{remark}

\begin{definition} A function $f:\Xb_{d,0} \rightarrow \Rb$ is called \emph{intrinsic} if  $f(\Omega_1, p_1) = f(\Omega_2, p_2)$ whenever there exists a biholomorphism $\varphi: \Omega_1 \rightarrow \Omega_2$ with $\varphi(p_1) = p_2$. 
\end{definition}

\begin{example} There are many examples of intrinsic functions, for instance the functions:
\begin{align*}
(\Omega, x) \rightarrow s_\Omega(x)
\end{align*}
and 
\begin{align*}
(\Omega, x) \rightarrow \max_{v \in T_x \Omega \setminus \{0\}} \abs{H_{b_\Omega}(v) - \frac{-4}{d+1}}
\end{align*}
are intrinsic.
\end{example}

Using the fact that the unit ball is a homogeneous domain we have the following:

 \begin{observation} 
 If $\Bc \subset \Cb^d$ is the unit ball and $f:\Xb_{d,0} \rightarrow \Rb$ is an intrinsic function, then $f(\Bc, x) = f(\Bc, 0)$ for all $x \in \Bc$. 
 \end{observation}
 
The set $\Xb_{d,0}$ has a topology coming from the local Hausdorff topology (see Section~\ref{sec:space-convex} below) and when a intrinsic function is continuous in this topology we obtain a generalized version of Klembeck's Theorem for convex domains:
 
 \begin{proposition}\label{prop:gen_Klembeck}(see Section~\ref{sec:klembeck})
 Suppose $f:\Xb_{d,0} \rightarrow \Rb$ is a continuous intrinsic function and $\Omega$ is a bounded convex domain with $C^2$ boundary. If $\xi \in \partial \Omega$ is a strongly pseudoconvex point of $\partial\Omega$, then 
\begin{align*}
\lim_{z \rightarrow \xi} f(\Omega, z) = f(\Bc, 0).
\end{align*}
\end{proposition}

The main result of this paper is the following two converses to the above proposition: 

\begin{theorem}\label{thm:cont}(see Section~\ref{sec:main_result})
Suppose that $f:\Xb_{d,0} \rightarrow \Rb$ is a continuous intrinsic function with the following property: if $\Omega \in \Xb_d$ and $f(\Omega, x) = f(\Bc, 0)$ for all $x \in \Omega$, then $\Omega$ is biholomorphic to $\Bc$. 

Then there exists some $\epsilon=\epsilon(d,f) > 0$ so that:  if $\Omega \subset \Cb^d$ is a bounded convex domain with $C^\infty$ boundary and 
\begin{align*}
\abs{f(\Omega, z) -f(\Bc,0)} \leq \epsilon
\end{align*}
outside some compact subset of $\Omega$, then $\Omega$ is strongly pseudoconvex and thus 
\begin{align*}
\lim_{z \rightarrow \partial \Omega} f(\Omega, z) = f(\Bc,0).
\end{align*}
\end{theorem}

Some interesting intrinsic functions, for instance the squeezing function, do not appear to be continuous on $\Xb_{d,0}$ but are upper-semicontinuous. So we will also establish the following: 

\begin{theorem}\label{thm:upper_cont}(see Section~\ref{sec:main_result})
Suppose that $f:\Xb_{d,0} \rightarrow \Rb$ is an upper semi-continuous intrinsic function with the following property: if $\Omega \in \Xb_d$ and $f(\Omega, x) \geq f(\Bc, 0)$ for all $x \in \Omega$, then $\Omega$ is biholomorphic to $\Bc$. 

Then there exists some $\epsilon=\epsilon(d,f) > 0$ so that: if $\Omega \subset \Cb^d$ is a bounded convex domain with $C^\infty$ boundary and 
\begin{align*}
f(\Omega, z) \geq f(\Bc,0) - \epsilon
\end{align*}
outside some compact subset of $\Omega$, then $\Omega$ is strongly pseudoconvex.
\end{theorem}

\subsection{Outline of proof} 

The key step in the proof of Theorem~\ref{thm:cont} and~\ref{thm:upper_cont} is establishing the following:

\begin{proposition}(see Theorems~\ref{thm:rescale_finite_type}, ~\ref{thm:rescale_inf_type}, and~\ref{thm:L_is_cpct})
For any $d > 0$, there exists a compact set $\Lb \subset \Xb_{d}$ with the following properties:
\begin{enumerate}
\item if $\Omega \in \Lb$, then $\Omega$ is not biholomorphic to the unit ball,
\item if $\Omega \subset \Cb^d$ is a bounded convex domain with $C^\infty$ boundary which is not strongly pseudoconvex, then there exists a sequence $x_n \in \Omega$ approaching the boundary and affine maps $A_n \in \Aff(\Cb^d)$ so that $A_n(\Omega, x_n)$ converges to $(\Omega_\infty, 0)$ where $ \Omega_\infty \in \Lb$. 
\end{enumerate}
\end{proposition}

One then deduces Theorem~\ref{thm:cont} from the Proposition by contradiction: assume for a contradiction that for any $n > 0$ there exists $\Omega_n \subset \Cb^d$ a bounded convex domain with $C^\infty$ boundary so that 
\begin{enumerate}
\item outside some compact set of $\Omega$
\begin{align*}
\abs{f(\Omega_n, z) -f(\Bc,0)} \leq 1/n,
\end{align*}
\item $\Omega_n$ is not strongly pseudoconvex. 
\end{enumerate}
Then for each $n >0$, there exists points $x_m^{(n)} \in \Omega_n$ and affine maps $A^{(n)}_m \in \Aff(\Cb^d)$ so that $A_{n,m}(\Omega_n, x_{n,m})$ converges to $(\wh{\Omega}_n, \wh{x}_n)$ and $\wh{\Omega}_n \in \Lb$. A simple argument will show that 
\begin{align*}
\abs{f(\wh{\Omega}_n, z) -f(\Bc,0)} \leq 1/n
\end{align*}
for all $z \in \wh{\Omega}_n$. Now using the fact that $\Lb$ is compact we can pass to a subsequence and assume that $\wh{\Omega}_n$ converges to some $\wh{\Omega} \in \Lb$. Now 
\begin{align*}
f(\wh{\Omega}, z) =f(\Bc,0)
\end{align*}
for all $z \in \wh{\Omega}$. So $\wh{\Omega}$ is biholomorphic to $\Bc$, which is a contradiction since $\wh{\Omega} \in \Lb$.  

The proof of Theorem~\ref{thm:upper_cont} is nearly identical. 

\subsection{Notations} 

We end the introduction by fixing some very basic  notations. 
\begin{enumerate}
\item For $v,w \in\Cb^d$, $\ip{v,w}$ will denote the standard Hermitian inner product and $\norm{v}$ will denote the standard Euclidean distance.
\item For a domain $\Omega \subset \Cb^d$ we will let $b_\Omega$ denote the Bergman metric, $k_\Omega$ denote the Kobayashi infinitesimal metric, and $K_\Omega$ denote the Kobayashi distance. 
\item $\Aff(\Cb^d)$ will denote the group of affine automorphisms of $\Cb^d$. 
\end{enumerate}

\subsection*{Acknowledgements} 

I would like to thank the referee for a number of comments and corrections which improved the present work. This material is based upon work supported by the National Science Foundation under Grant Number NSF 1400919.

\section{The space of convex sets} \label{sec:space-convex}

\subsection{The local Hausdorff topology} Given a set $A \subset \Cb^d$, let $\Nc_{\epsilon}(A)$ denote the \emph{$\epsilon$-neighborhood of $A$} with respect to the Euclidean distance. The \emph{Hausdorff distance} between two compact sets $A,B \subset \Cb^d$ is given by
\begin{align*}
d_{H}(A,B) = \inf \left\{ \epsilon >0 : A \subset \Nc_{\epsilon}(B) \text{ and } B \subset \Nc_{\epsilon}(A) \right\}.
\end{align*}
Equivalently, 
\begin{align*}
d_H(A,B) = \max\left\{\sup_{a \in A}\inf_{b \in B} \norm{a-b}, \sup_{b \in B} \inf_{a \in A} \norm{a-b} \right\}.
\end{align*}
The Hausdorff distance is a complete metric on the space of compact sets in $\Cb^d$.

The space of all closed convex sets in $\Cb^d$ can be given a topology from the local Hausdorff semi-norms. For $R >0$ and a set $A \subset \Cb^d$ let $A^{(R)} := A \cap B_R(0)$. Then define the \emph{local Hausdorff semi-norms} by
\begin{align*}
d_H^{(R)}(A,B) := d_H(A^{(R)}, B^{(R)}).
\end{align*}
Since an open convex set is completely determined by its closure, we say a sequence of open convex sets $A_n$ converges in the local Hausdorff topology to an open convex set $A$ if there exists some $R_0 \geq 0$ so that 
\begin{align*}
\lim_{n \rightarrow \infty} d_H^{(R)}(\overline{A}_n,\overline{A}) = 0
\end{align*}
for all $R \geq R_0$. 

Finally we introduce a topology on $\Xb_d$ and $\Xb_{d,0}$ using the local Hausdorff topology: 
\begin{enumerate}
\item A sequence $\Omega_n$ converges to $\Omega_\infty$ in $\Xb_d$ if $\Omega_n \rightarrow \Omega_\infty$ in the local Hausdorff topology. 
\item A sequence $(\Omega_n, x_n)$ converges to $(\Omega_\infty, x_\infty)$ in $\Xb_{d,0}$ if $\Omega_n \rightarrow \Omega_\infty$ in the local Hausdorff topology and $x_n \rightarrow x_\infty$.
\end{enumerate} 

\subsection{Continuity of the Kobayashi distance and metric}

Unsurprisingly, the Kobayashi distance is continuous with respect to the local Hausdorff topology.

\begin{theorem}
\label{thm:dist_conv}
Suppose $\Omega_n$ converges to $\Omega$ in $\Xb_{d}$. Then 
\begin{align*}
K_{\Omega}(x,y) = \lim_{n \rightarrow \infty} K_{\Omega_n}(x,y)
\end{align*}
for all $x,y \in \Omega$ uniformly on compact sets of $\Omega \times \Omega$.
\end{theorem}

See Theorem 4.1 in ~\cite{Z2016} for a detailed argument. The proof of Theorem 4.1 in~\cite{Z2016} also implies that the Kobayashi metric is continuous:

\begin{theorem}\label{thm:metric_conv}
Suppose $\Omega_n$ converges to $\Omega$ in $\Xb_{d}$. Then 
\begin{align*}
k_{\Omega}(x;v) = \lim_{n \rightarrow \infty} k_{\Omega_n}(x;v)
\end{align*}
for all $x \in \Omega$ and $v \in \Cb^d$ uniformly on compact sets of $\Omega \times \Cb^d$.
\end{theorem}

\subsection{The affine group acts co-compactly}\label{subsec:acts_cocompact}

Let $\Aff(\Cb^d)$ be the group of complex affine isomorphisms of $\Cb^d$. Then $\Aff(\Cb^d)$ acts on $\Xb_d$ and $\Xb_{d,0}$. Remarkably, the action of $\Aff(\Cb^d)$ on $\Xb_{d,0}$ is co-compact:

\begin{theorem}[Frankel \cite{F1989b}]\label{thm:frankel_compactness}
The group $\Aff(\Cb^d)$ acts co-compactly on $\Xb_{d,0}$, that is there exists a compact set $K \subset \Xb_{d,0}$ so that $\Aff(\Cb^d) \cdot K = \Xb_{d,0}$. 
\end{theorem}

This subsection is devoted to producing a particular compact set whose $\Aff(\Cb^d)$-translates cover $\Xb_{d,0}$.

Let $e_1, \dots, e_d$ be the standard basis of $\Cb^d$ and for $1 \leq i \leq d$ define the complex $(d-i)$-plane $Z_i$ by 
\begin{align*}
Z_i = \left\{ e_i + \sum_{j=1}^{d-i} z_j e_{i+j} : z_1, \dots, z_{d-i} \in \Cb\right\}.
\end{align*}

\begin{definition}
Let $\Kb \subset \Xb_{d,0}$ be the set of pairs $(\Omega, 0)$ where $\Omega$ is a convex domain so that:
\begin{enumerate} 
\item $\Db e_i \subset \Omega$ for each $1 \leq i \leq d$, 
\item $Z_i \cap \Omega=\emptyset$ for each $1 \leq i \leq d$. 
\end{enumerate}
\end{definition} 

\begin{theorem}\label{thm:weak_fund_domain} With the notation above:
\begin{enumerate}
\item  $\Kb$ is a compact subset of $\Xb_{d,0}$.
\item For any $(\Omega, x) \in \Xb_{d,0}$ there exists some $A \in \Aff(\Cb^d)$ so that $A(\Omega, x) \in \Kb$. 
\end{enumerate}
\end{theorem}

We begin the proof of Theorem~\ref{thm:weak_fund_domain} by proving the following:

\begin{lemma}\label{lem:K_defn} Suppose $\Omega \subset \Cb^d$ is a convex domain so that:
\begin{enumerate} 
\item $\Db e_i \subset \Omega$ for each $1 \leq i \leq d$, 
\item $Z_i \cap \Omega=\emptyset$ for each $1 \leq i \leq d$. 
\end{enumerate}
Then $\Omega$ contains no complex affine lines (and hence is in $\Kb$). 
\end{lemma}

\begin{proof}
Since $\Omega$ is convex, for each $1 \leq i \leq d$ there exists a real hyperplane $H_{i}$ so that $Z_i \subset H_{i}$ and $H_{i} \cap \Omega = \emptyset$. Then for each $1 \leq i \leq d$ there exists some $v_i \in \Cb^d$ and $a_i \in \Rb$ so that 
\begin{align*}
H_i = \{ z \in \Cb^d : \Real \ip{v_i,z}=a_i\}.
\end{align*}
Since $0 \notin H_i$, we see that $a_i \neq 0$ and so we can assume that $a_i = 1$. Now suppose that $v_i=(v_{i,1}, \dots, v_{i,d})$. Then since $Z_i \subset H_i$, we see that $v_{i,j}=0$ for $j > i$ and $v_{i,i} =1$. Thus $v_1, \dots, v_d$ forms a $\Cb$-basis for $\Cb^d$. 

Now suppose that $L$ is a complex line, we claim that $L$ is not contained in $\Omega$. Fix $a, b \in \Cb^d$ so that
\begin{align*}
L = \{ b + az : z \in \Cb\}.
\end{align*}
 Then, since $v_1, \dots, v_d$ is a basis of $\Cb^d$, there exists $1 \leq i \leq d$ so that $\ip{v_i, a} \neq 0$. But then $L \cap H_i \neq \emptyset$ and so $L$ is not contained in $\Omega$. Since $L$ was an arbitrary complex line, we see that $\Omega$ contains no complex affine lines. 
 \end{proof}

\begin{proof}[Proof of Theorem~\ref{thm:weak_fund_domain}]
Suppose $(\Omega_n,0)$ is a subsequence in $\Kb$. By passing to a subsequence we can assume that $\overline{\Omega}_{n_k}$ converges in the local Hausdorff topology to some closed convex set $\Cc \subset \Cb^d$. Then $\Db e_i \subset \Cc$ for $1 \leq i \leq d$ and so $\Cc$ has non-empty interior. Let $\Omega$ be the interior of $\Cc$. Then $\Omega_n$ converges to $\Omega$ in the local Hausdorff topology. Moreover, 
\begin{enumerate} 
\item $\Db e_i \subset \Omega$ for each $1 \leq i \leq d$, 
\item $Z_i \cap \Omega=\emptyset$ for each $1 \leq i \leq d$. 
\end{enumerate}
So by Lemma~\ref{lem:K_defn}, $\Omega \in \Kb$ and thus $(\Omega_n,0)$ converges to $(\Omega, 0)$ in $\Kb$. 

Now suppose that $(\Omega, x) \in \Xb_{d,0}$. We will find some $A \in \Aff(\Cb^d)$ so that $A(\Omega, x) \in \Kb$. Pick $y_1, \dots, y_d \in \partial \Omega$ as follows: first let $y_1 \in \partial \Omega$ be the closest point to $x$ in $\partial \Omega$. Then supposing $y_1, \dots, y_{k-1}$ have already been selected, let $W$ be the maximal complex subspace through $x$ which is orthogonal to the complex lines 
\begin{align*}
x+\Cb(y_1-x), \dots, x+\Cb(y_{k-1}-x)
\end{align*}
and let $y_k$ be closest point to $x$ in $\partial \Omega \cap W$. 

Next let $T:\Cb^d \rightarrow \Cb^d$ be the translation $T(z) = z-x$ and let $U : \Cb^d \rightarrow \Cb^d$ be the unitary map defined by
\begin{align*}
(UT)(y_i) = \norm{x-y_i}e_i.
\end{align*}
Next consider the diagonal matrix 
\begin{align*}
\Lambda = \begin{pmatrix} 1/\norm{x-y_1} & & \\ & \ddots & \\ & & 1/\norm{x-y_d} \end{pmatrix}
\end{align*}
and the affine map $A=\Lambda U T$. Then by construction, 
\begin{enumerate} 
\item $\Db e_i \subset A\Omega$ for each $1 \leq i \leq d$, 
\item $Z_i \cap (A\Omega)=\emptyset$ for each $1 \leq i \leq d$. 
\end{enumerate}
Hence $A (\Omega,x) \in \Kb$. 

\end{proof}

\section{Rescaling convex domains}\label{sec:limit_domains}

Suppose $\Omega \subset \Cb^d$ is a convex domain and $x_n \in \Omega$ is a sequence converging to the boundary. Then by Theorem~\ref{thm:weak_fund_domain} it is always possible to find affine maps $A_n$ so that (after possibly passing to a subsequence) $A_n(\Omega, x_n)$ converges in $\Xb_{d,0}$ to some pair $(\Omega_\infty, x_\infty)$. In this section we describe some of the possible limit domains $\Omega_\infty$ when $\Omega$ has $C^\infty$ boundary. 

In the following it will be notationally convenient to work in $\Cb^{d+1}$ instead of $\Cb^d$. 

\subsection{Line type} In this subsection we recall the line type of a point in the boundary of a convex domain. This discussion is based on a paper of Yu~\cite{Y1992}.

Suppose $U \subset \Cb^d$ is a neighborhood of $0$ and $f: U \rightarrow \Rb$ is a $C^\infty$ function with $f(0)=0$. If $v \in \Cb^d$, let $U_v = \{ z \in \Cb : zv \in U\}$ and let $f_v : U_v \rightarrow \Cb$ be the function defined by $f_v(z)= f(zv)$. Finally let $\ord(f;v) \in \Nb \cup \{\infty\}$ denote the order of vanishing of the function $f_v$ at $z=0$.

\begin{observation} With the notation above and $r \in \Nb$, the following are equivalent: 
\begin{enumerate}
\item $\ord(f;v) \geq r$, 
\item $\frac{\partial^{k+\ell}}{\partial z^{k} \partial \overline{z}^{\ell}} f_v(0) = 0$ for all $k+\ell \leq r-1$, 
\item there exists $C > 0$ so that $\abs{f_v(z)} \leq C \abs{z}^r$ for $z$ sufficiently close to $0$.
\end{enumerate} 
\end{observation} 

Now suppose that $f: U \subset \Cb^d \rightarrow \Rb$ is $C^\infty$, convex, non-negative, and $f(0)=0$. For any $r \in \Nb \cup \{\infty\}$ let 
\begin{align*}
V_r = \{ v \in \Cb^d : \ord(f;v) \geq r\}.
\end{align*}
Clearly $V_\infty \subset V_r$ and $V_{r+1} \subset V_r$ for all $r \in \Nb$. 

\begin{observation} With the notation above, $V_r$ is a complex linear subspace of $\Cb^d$. \end{observation}

\begin{proof} First suppose that $r \in \Nb$. If $v \in V_r$, then clearly $\lambda v \in V_r$ for $\lambda \in \Cb$. Moreover, if $v,w \in V_r$ then there exists $C>0$ so that 
\begin{align*}
\abs{f(zv)}, \abs{f(zw)} \leq C\abs{z}^r
\end{align*}
for $z$ sufficiently small. Then for $z$ sufficiently small we have
\begin{align*}
\abs{f(z(v+w))} = f(zv+zw) \leq \frac{1}{2} f(2zv) + \frac{1}{2} f(2zw) \leq 2^{r} C \abs{z}^r
\end{align*}
since $f$ is convex and non-negative. 

Next consider the case in which $r=\infty$. Then $V_\infty = \cap_{r \in \Nb} V_r$ and so $V_\infty$ is also a subspace.
\end{proof}

Then there exists $r_1 < r_2 < \dots < r_k$ in $\Nb \cup \{\infty\}$ so that 
\begin{enumerate}
\item $r_1=\max\left\{ r \in \Nb \cup \{\infty\} : V_r = \Cb^d\right\}$,
\item $r_k = \max\left\{ r \in \Nb \cup \{\infty\} : V_{r} \neq (0)\right\}$, 
\item $V_{r_{k}} \subsetneq V_{r_{k-1}} \subsetneq \dots \subsetneq V_{r_1}$, and
\item $\ord(f;v) = r_i$ for all $v \in V_{r_{i}} \setminus V_{r_{i+1}}$.
\end{enumerate}
Now for $0 \leq i \leq k$ let 
\begin{align*}
d_i = \left\{ \begin{array}{ll} 
0 & \text{ if } i=0 \\
\dimension V_{r_i} & \text{ if } 1 \leq i \leq k.
\end{array}
\right.
\end{align*}
Finally define the \emph{type of $f$ at $0$} to be the $d$-tuple $m=(m_1, \dots, m_d)$ where 
\begin{align*}
m_j = r_i \quad \text{ if } \quad d_{i-1} < j \leq d_i.
\end{align*}

Now suppose that $\Omega \subset \Cb^{d+1}$ is a convex domain with $C^\infty$ boundary. Then we can associated to every point $\xi \in \partial \Omega$ a tuple $m(\xi)=(m_1, \dots, m_d)$ where $m_i \in \Nb \cup \{\infty\}$ as follows. Using an affine transformation we can assume that $\xi=0$, $T_{\xi} \partial \Omega = \Rb \times \Cb^d$, and $\Omega \subset \{ (z_1, \dots, z_{d+1}) \in \Cb^{d+1} : \Imaginary(z_1) > 0\}$. Now there exists neighborhoods $U \subset \Rb$, $V \subset \Rb$, $W \subset \Cb^d$, and a $C^\infty$ function $F: U \times W \rightarrow V$ so that 
\begin{align*}
\Omega \cap \Big( (U+iV) \times W\Big) = \left\{ (x+iy, z) : y > F(x, z)\right\}.
\end{align*}
Then the function $f=F|_{\{0\} \times \Cb^d} : W \rightarrow \Rb$ is $C^\infty$, convex, non-negative, and $f(0)=0$. Finally, let $m(\xi)$ be the type of $f$ at $z=0$.

\begin{definition} Suppose that $\Omega \subset \Cb^{d+1}$ is a convex domain with $C^\infty$ boundary. The \emph{line type} of $\xi \in \partial \Omega$ is the $d$-tuple $m(\xi) \in ( \Nb \cup \{\infty\})^d$. 
\end{definition}

\begin{remark}\
\begin{enumerate}
\item In the definition of $m(\xi)$, we made one choice: the initial affine transformation. However the value of $m(\xi)$ does not depend on this choice, see for instance~\cite[Proposition 2]{Y1992}.
\item $\Omega$ is strongly pseudoconvex at $\xi$ if and only if $m(\xi) = (2,\dots, 2)$. 
\end{enumerate}
\end{remark} 

\subsection{The finite type case}

Let $e_0,e_1 \dots, e_d$ be the standard basis of $\Cb^{d+1}$ and for $0 \leq i \leq d$ define the complex $(d-i)$-plane $Z_i$ by 
\begin{align*}
Z_i = \left\{ e_i + \sum_{j=1}^{d-i} z_j e_{i+j} : z_1, \dots, z_{d-i} \in \Cb\right\}.
\end{align*}

\begin{theorem}\label{thm:rescale_finite_type} 
Suppose that $\Omega \subset \Cb^{d+1}$ is a bounded convex domain with $C^\infty$ boundary, $\xi \in \partial \Omega$, and $m(\xi)=(m_1, \dots, m_d) \in \Nb^d$. Then there exists $x_n \in \Omega$ converging to $\xi$ and affine maps $A_n \in \Aff(\Cb^{d+1})$ so that 
\begin{align*}
A_n(\Omega, x_n) \rightarrow (\wh{\Omega}, 0) \text{ in } \Xb_{d+1,0}
\end{align*}
where 
\begin{align*}
\wh{\Omega} = \left\{ (x+iy, z) : x < 1- P(z) \right\}
 \end{align*}
and
 \begin{enumerate}
 \item  $P: \Cb^d \rightarrow \Rb$ is a convex, non-negative polynomial with $P(0)=0$, 
 \item  for all $t > 0$ and $(z_1, \dots, z_d) \in \Cb^d$ we have
 \begin{align*}
 P(t^{1/m_1}z_1, \dots, t^{1/m_d}z_d) = t P(z_1, \dots, z_d).
 \end{align*}
 \item $ \Db e_i \subset \wh{\Omega}$ for all $0 \leq i \leq d$. 
 \item $Z_i \cap \wh{\Omega} = \emptyset$ for all $0 \leq i \leq d$. 
  \end{enumerate}
 \end{theorem}
 
 \begin{remark} The last two conditions imply that $\wh{\Omega}$ is contained in the compact set $\Kb$ defined in Subsection~\ref{subsec:acts_cocompact}. \end{remark}
  
 \begin{proof}
 By Proposition 2 in~\cite{G1997}, there exists $x_n \in \Omega$ converging to $\xi$ and affine maps $B_n \in \Aff(\Cb^{d+1})$ so that 
\begin{align*}
B_n(\Omega, x_n) \rightarrow (\wh{\Omega}_0, 0) \text{ in } \Xb_{d+1,0}
\end{align*}
where 
\begin{align*}
\wh{\Omega}_0 = \left\{ (x+iy, z) : x < 1- P_0(z) \right\}
 \end{align*}
and
 \begin{enumerate}
 \item  $P_0: \Cb^d \rightarrow \Rb$ is a convex, non-negative polynomial with $P_0(0)=0$, 
 \item  for all $t > 0$ and $(z_1, \dots, z_d) \in \Cb^d$ we have
 \begin{align*}
 P_0(t^{1/m_1}z_1, \dots, t^{1/m_d}z_d) = t P_0(z_1, \dots, z_d).
 \end{align*}
 \end{enumerate}
 
The rest of the argument is devoted to producing an affine map $A$ so that $A_n = AB_n$ satisfies the rest of the conditions in the Theorem. 

Pick $0 = d_0 < d_1 < \dots < d_{r+1} = d$ so that: 
\begin{align*}
m_1 = \dots = m_{d_1} < m_{d_1+1} = \dots = m_{d_2} < \dots < m_{d_{r-1}+1} = \dots = m_{d}.
\end{align*}
Then define vector spaces
 \begin{align*}
 W_\ell = \Span_{\Cb} \{ e_i : d_\ell < i \leq d_{\ell+1}\}.
 \end{align*}
Now for each $0 \leq \ell \leq r$ we select an orthogonal basis $x_{d_\ell+1}, x_{d_\ell+2}, \dots, x_{d_{\ell+1}} \in W_\ell$ as follows: first pick $x_{d_\ell+1} \in W_\ell$ so that  $x_{d_\ell+1}$ is the closest point in $\partial \wh{\Omega}_0 \cap W_\ell$ to $0$. Then assuming $x_{d_\ell+1}, \dots, x_{d_\ell+k-1}$ have been selected pick  $x_{d_\ell+k} \in W_\ell$ so that $x_{d_\ell+k}$ is a closest point in 
\begin{align*}
\partial \wh{\Omega}_0 \cap \Span_{\Cb} \{ x_{d_\ell+1}, \dots, x_{d_\ell+k-1}\}^{\bot} \cap W_\ell
\end{align*}
to $0$.
 
Let $U : \Cb^{d+1} \rightarrow \Cb^{d+1}$ be the unitary map with $U(e_0) = e_0$ and $U(x_{k}) = \norm{x_k}e_k$ for $1 \leq k \leq d$. Next let 
\begin{align*}
\Lambda = 
\begin{pmatrix}
1 & & & \\ 
& 1/\norm{x_1} & & \\
& & \ddots & \\
& & & 1/\norm{x_d} 
\end{pmatrix}.
\end{align*}
Finally let $A = \Lambda U$ and consider $\wh{\Omega} =A\wh{\Omega}_0$. By construction 
\begin{align*}
\wh{\Omega} = \left\{ (x+iy, z) : x < 1 + P(z) \right\}
 \end{align*}
where $P: \Cb^d \rightarrow \Rb$ is a convex, non-negative polynomial with $P(0)=0$. Moreover, since $U$ and $\Lambda$ preserve the subspaces $W_\ell$ we see that 
 \begin{align*}
 P(t^{1/m_1}z_1, \dots, t^{1/m_d}z_d) = t P(z_1, \dots, z_d).
 \end{align*}
for all $t > 0$ and $(z_1, \dots, z_d) \in \Cb^d$. By the way we picked the $x_i$, we see that $e_i \in \partial \wh{\Omega}$ and $ \Db e_i \subset \wh{\Omega}$ for all $1 \leq i \leq d$. 

We now show that 
\begin{align*}
Z_i \cap \wh{\Omega} = \emptyset
\end{align*}
for all $0 \leq i \leq d$. By construction $Z_0 \cap \wh{\Omega} = \emptyset$. Now since $\partial \wh{\Omega}$ is $C^1$ (actually $C^\infty$) 
each $\xi \in \partial \wh{\Omega}$ is contained in a unique tangent real hyperplane $H$ and since $\wh{\Omega}$ is convex $H \cap \wh{\Omega} = \emptyset$. Now this hyperplane is given by: 
\begin{align*}
H = \xi + \left\{ z \in \Cb^{d+1} : \Real\left( z_0 + 2 \sum_{j=1}^d \frac{\partial P}{\partial z_j}(\xi)z_j \right) = 0 \right\}. 
\end{align*}
Thus it is enough to show that 
\begin{align*}
\frac{\partial P}{\partial z_j}(e_i)=0
\end{align*}
for all $1 \leq i < j \leq d$.

Fix $1 \leq i < d$ and assume that $d_\ell < i \leq d_{\ell+1}$. If $d_{\ell} < i < j \leq d_{\ell+1}$ then the way we selected $x_1, \dots, x_d$ implies that 
\begin{align*}
\frac{\partial P}{\partial z_j}(e_i)=0.
\end{align*}
In the case in which $i \leq d_{\ell+1} < j$ we have $m_i < m_j$. Since 
 \begin{align*}
 P(t^{1/m_1}z_1, \dots, t^{1/m_d}z_d) = t P(z_1, \dots, z_d)
 \end{align*}
 we see that 
 \begin{align}
 \label{eq:one}
t^{1/m_j}\frac{\partial P}{\partial z_j}(t^{1/m_i}e_i) = t\frac{\partial P}{\partial z_j}(e_i).
\end{align}
Now suppose for a contradiction that $\frac{\partial P}{\partial z_j}(e_i) \neq 0$. Then 
\begin{align*}
s \in \Rb \rightarrow \frac{\partial P}{\partial z_j}(s e_i)
\end{align*}
is a non-zero polynomial in $s$. Suppose this polynomial has degree $D$. Then comparing sides in Equation~\ref{eq:one} we see that 
\begin{align*}
\frac{1}{m_j} + \frac{D}{m_i} = 1.
\end{align*}
So $D = m_i - m_i/m_j$. But $D$ is an integer and $m_j > m_i$, so we have a contradiction. Thus 
 \begin{align*}
\frac{\partial P}{\partial z_j}(e_i)=0
\end{align*}
for all $1 \leq i < j \leq d$. And so 
\begin{align*}
Z_i \cap \wh{\Omega} = \emptyset
\end{align*}
for all $0 \leq i \leq d$. 

\end{proof}
 
 \subsection{The infinite type case}
 
 \begin{theorem}\label{thm:rescale_inf_type}
Suppose that $\Omega \subset \Cb^{d+1}$ is a bounded convex domain with $C^\infty$ boundary, $\xi \in \partial \Omega$, and $m(\xi) \notin \Nb^d$. Then there exists $x_n \in \Omega$ converging to $\xi$ and affine maps $A_n \in \Aff(\Cb^{d+1})$ so that 
\begin{align*}
A_n(\Omega, x_n) \rightarrow (\wh{\Omega}, 0) \text{ in } \Xb_{d+1,0}
\end{align*}
where 
\begin{enumerate}
\item $\Db e_i \subset \wh{\Omega}$ for $0 \leq i \leq d$,
\item $Z_i \cap \wh{\Omega} = \emptyset$ for $0 \leq i \leq d$, 
\item $e_0 + \Db e_1 \subset \partial \wh{\Omega}$.
\end{enumerate}
\end{theorem}

\begin{proof}
By (the proof of) Proposition 6.1 in~\cite{Z2016}, there exists $x_n \in \Omega$ converging to $\xi$ and affine maps $B_n \in \Aff(\Cb^{d+1})$ so that 
\begin{align*}
B_n(\Omega, x_n) \rightarrow (\wh{\Omega}_0, 0) \text{ in } \Xb_{d+1,0}
\end{align*}
where 
\begin{enumerate}
\item $\Db e_0 \cup \Db e_1 \in \wh{\Omega}_0$,
\item $(e_0 + \Db e_1) \cup\{ e_1 \}\subset \partial \wh{\Omega}_0$.
\end{enumerate}

The rest of the argument is devoted to producing an affine map $A$ so that $A_n = AB_n$ satisfies the rest of the conditions in the Theorem. 

We begin by selecting $y_0, \dots, y_d \in \partial \wh{\Omega}_0$ and subspaces $W_0, \dots, W_d \subset \Cb^{d+1}$ as follows: first let $y_0 = e_0$ and let $W_0$ be a complex hyperplane so that:
\begin{enumerate}
\item $\Db e_1 \subset W_0$, 
\item $(y_0 + W_0) \cap \wh{\Omega}_0 = \emptyset$. 
\end{enumerate}
Since $\wh{\Omega}_0$ is convex, such a hyperplane exists. Next pick $y_1 \in \Cb e_1 \cap \partial \wh{\Omega}_0$ so that 
\begin{align*}
\norm{y_1} = \inf\{\norm{z} : z \in \Cb e_1 \cap \partial \wh{\Omega}_0\}.
\end{align*}
Since $e_1 \in \partial \wh{\Omega}_0$, we have that $\norm{y_1} \leq 1$. Then let $W_1$ be a complex subspace so that 
\begin{enumerate}
\item $\dimension_{\Cb} W_1 = \dimension_{\Cb} W_0-1 = d-1$, 
\item $W_1 \subset W_0$, 
\item  $(y_1 + W_1) \cap \wh{\Omega}_0 = \emptyset$. 
\end{enumerate}
Now supposing that $y_1, \dots, y_k \in  \partial \wh{\Omega}_0$ and $W_0, \dots, W_k \subset \Cb^{d+1}$ have already been selected, select $y_{k+1} \in W_k \cap \partial \wh{\Omega}_0$ so that 
\begin{align*}
\norm{y_{k+1}} = \inf\{\norm{z} : z \in W_k \cap \partial \wh{\Omega}_0\}
\end{align*}
and let $W_{k+1}$ be a complex subspace so that 
\begin{enumerate}
\item $\dimension_{\Cb} W_{k+1} = \dimension_{\Cb} W_{k} -1 = d-k-1$, 
\item $W_{k+1} \subset W_k$, 
\item  $(y_{k+1} + W_{k+1}) \cap \wh{\Omega}_0 = \emptyset$. 
\end{enumerate}

Next let $A \in \GL_{d+1}(\Cb)$ be the linear map so that 
\begin{align*}
A(y_i) = e_i \text{ for all } 0 \leq i \leq d.
\end{align*}
Then 
\begin{align*}
A(W_i) = \Span_{\Cb}\{e_{i+1}, \dots, e_{d}\}.
\end{align*}
So if $\wh{\Omega} := A\wh{\Omega}_0$ and $A_n:=AB_n$, then 
\begin{align*}
A_n(\Omega, x_n) \rightarrow (\wh{\Omega}, 0) \text{ in } \Xb_{d+1,0}
\end{align*}
where 
\begin{enumerate}
\item $\Db e_i \subset \wh{\Omega}$ for $0 \leq i \leq d$,
\item $Z_i \cap \wh{\Omega} = \emptyset$ for $0 \leq i \leq d$, 
\item $e_0 + \Db e_1 \subset \partial \wh{\Omega}$.
\end{enumerate}
\end{proof}

\section{The rescaled domains}

In this section we show that the domains obtained by rescaling in Section~\ref{sec:limit_domains} are contained in a compact subset of $\Xb_{d+1}$. 

Let $\Fb \subset \Xb_{d+1}$ be the set of domains $\Omega \in \Xb_{d+1}$ with the following properties 
\begin{enumerate}
\item $\Db e_i \subset \Omega$ for $0 \leq i \leq d$, 
\item $Z_i \cap \Omega = \emptyset$ for $0 \leq i \leq d$, 
\item $e_0 + \Db e_1 \subset \partial \Omega$.
\end{enumerate}
Next suppose that $m=(m_1, \dots, m_d)$ and $2 \leq m_1 \leq m_2 \leq \dots \leq m_d<\infty$. Let $\Pb(m) \subset \Xb_{d+1}$ be the set of domains $\Omega$ where
\begin{align*}
\Omega= \left\{ (x+iy, z) : x < 1- P(z) \right\}
 \end{align*}
and
 \begin{enumerate}
 \item  $P: \Cb^d \rightarrow \Rb$ is a convex, non-negative polynomial with $P(0)=0$, 
 \item  for all $t > 0$ and $(z_1, \dots, z_d) \in \Cb^d$ we have
 \begin{align*}
 P(t^{1/m_1}z_1, \dots, t^{1/m_d}z_d) = t P(z_1, \dots, z_d),
 \end{align*}
 \item $ \Db e_i \subset \Omega$ for all $0 \leq i \leq d$, and
 \item $Z_i \cap \Omega = \emptyset$ for all $0 \leq i \leq d$. 
  \end{enumerate}
  Finally, define 
  \begin{align*}
  \Lb := \Fb \cup \cup_{m \neq  (2, \dots, 2) }  \Pb(m).
\end{align*}

\begin{theorem}\label{thm:L_is_cpct}
The subset $\Lb \subset \Xb_{d+1}$ is compact and if $\Omega \in \Lb$, then $\Omega$ is not biholomorphic to $\Bc$. 
\end{theorem}

\begin{proof} Let $\Fb_0 = \{ (\Omega, 0) : \Omega \in \Fb\}$. Then $\Fb_0$ is a closed subset of $\Kb$ and so Theorem~\ref{thm:weak_fund_domain} implies that $\Fb_0$ is compact. Hence $\Fb$ is compact. Thus, to show that $\Lb$ is compact it is enough to  consider a sequence $\Omega_n \in \Pb(m^{(n)})$ and show that there exists a subsequence $\Omega_{n_k}$ which converges to some domain in $\Lb$.  

Since $(\Omega_n, 0) \in \Kb$, we can pass to a subsequence so that $\Omega_n$ converges to some $\Omega$ in $\Xb_{d+1}$. 

Suppose that $P_n$ is the polynomial so that 
 \begin{align*}
 \Omega_n = \{ (x+iy, z) : x < 1- P_n(z) \}.
 \end{align*}
By passing to a subsequence we can suppose that 
 \begin{align*}
 m_i := \lim_{n \rightarrow \infty} m_i^{(n)}
 \end{align*}
 exists in $\Nb \cup \{\infty\}$ for all $1 \leq i \leq d$. \newline

\noindent \textbf{Case 1:}  $ m_d < \infty$. Then by passing to a subsequence we can assume that 
\begin{align*}
\left(m_1^{(n)}, \dots, m_d^{(n)}\right) =m=(m_1, \dots, m_d)
\end{align*}
for all $n \in \Nb$. This implies that  there exists an $N > 0$ so that 
\begin{align*}
P_n(z) = \sum_{2 \leq \abs{\alpha} + \abs{\beta} \leq N} c_{\alpha, \beta}^{(n)} z^{\alpha} \overline{z}^{\beta}.
\end{align*}

Now since $\Db e_i \subset \Omega_n$ we see that $P_n(ze_i) \leq 1$ for $1 \leq i \leq d$ and $\abs{z} \leq 1$. Then by convexity, $P_n(z) \leq 1$ for 
\begin{align*} 
z \in\Cc:= \mathrm{ConvexHull}\left( \Db e_1 \cup \dots \cup \Db e_d \right).
\end{align*}
Since $P_n$ is non-negative, we then see that 
\begin{align*}
\sup_{z \in \Cc} \abs{P_n(z)} \leq 1.
\end{align*}
Now by the equivalence of finite dimensional norms, there exists some $C >0$ so that
\begin{align*}
\sup_{2 \leq \abs{\alpha} + \abs{\beta} \leq N} \abs{  c_{\alpha, \beta}^{(n)}} \leq C \sup_{z \in \Cc} \abs{  \sum_{2 \leq \abs{\alpha} + \abs{\beta} \leq N} c_{\alpha, \beta}^{(n)} z^{\alpha} \overline{z}^{\beta}} \leq C
\end{align*}
Thus, after passing to a subsequence, we can suppose that 
\begin{align*}
\lim_{n \rightarrow \infty} c_{\alpha, \beta}^{(n)} = c_{\alpha, \beta}
\end{align*}
for all $2 \leq \abs{\alpha} + \abs{\beta} \leq N$. Then $P_n$ converges locally uniformly to the polynomial 
\begin{align*}
P(z) = \sum_{2 \leq \abs{\alpha} + \abs{\beta} \leq N} c_{\alpha, \beta} z^{\alpha} z^{\beta}.
\end{align*}
Then, since $\Omega_n$ converges to $\Omega$ in the local Hausdorff topology and 
\begin{align*}
\Omega_n = \{ (x+iy, z) : x < 1- P_n(z) \},
\end{align*}
we have
\begin{align*}
\Omega = \{ (x+iy, z) : x < 1- P(z) \}.
\end{align*}

We next claim that  $\Omega \in \Pb(m)$. Since the sequence $P_n$ converges locally uniformly to $P$ and each $P_n$ is convex and non-negative, we see that $P: \Cb^d \rightarrow \Rb$ is also convex and non-negative. Further, for $t > 0$ and $(z_1, \dots, z_d) \in \Cb^d$ we have
 \begin{align*}
 P(t^{1/m_1}z_1, \dots, t^{1/m_d}z_d) 
 & = \lim_{n \rightarrow \infty}  P_n(t^{1/m_1}z_1, \dots, t^{1/m_d}z_d) \\
 &  = \lim_{n \rightarrow \infty} t P_n (z_1, \dots, z_d) \\
 & = t P (z_1, \dots, z_d).
 \end{align*}
Since the sequence $\Omega_n$ converges to $\Omega$ in the local Hausdorff topology we see that 
\begin{enumerate}
 \item $ \Db e_i \subset \Omega$ for all $0 \leq i \leq d$ and 
 \item $Z_i \cap \Omega = \emptyset$ for all $0 \leq i \leq d$. 
  \end{enumerate}
Thus $\Omega \in \Pb(m)$. \newline

\noindent \textbf{Case 2:}  $ m_d= \infty$. Then using the fact that 
\begin{align*}
P_n(0, \dots, 0, z_d) = \abs{z_d}^{m_d^{(n)}} P\left(0, \dots, 0,  \frac{z_d}{\abs{z_d}} \right) \leq \abs{z_d}^{m_d^{(n)}} 
\end{align*}
we see that $e_0 + \Db e_d \subset \partial \Omega$. Thus $\Omega \in \Fb$. \newline

We now prove the second assertion of the theorem: \newline

\noindent \textbf{Claim:} If $\Omega \in \Lb$, then $\Omega$ is not biholomorphic to $\Bc$. \newline

If $\Omega \in \Pb(m)$ for some $m \neq (2,\dots, 2)$, then $\Omega$ is not biholomorphic to $\Bc$ by work of Coupet and Pinchuk~\cite{CP2001}. So suppose that $\Omega \in \Fb$ and let $K_\Omega$ be the Kobayashi distance on $\Omega$. Since $\partial \Omega$ contains a complex affine disk, Theorem 3.1 in~\cite{Z2016} implies that the metric space $(\Omega, K_\Omega)$ is not Gromov hyperbolic.  However, $\Bc$ endowed with its Kobayashi metric is a model of complex hyperbolic space which is Gromov hyperbolic. Since Gromov hyperbolicity is an isometric invariant, we see that $(\Omega, K_\Omega)$ and $(\Bc, K_{\Bc})$ cannot be isometric and thus $\Omega$  cannot be biholomorphic to $\Bc$. 

\end{proof}

\section{The proof of Proposition~\ref{prop:gen_Klembeck}}\label{sec:klembeck}

Proposition~\ref{prop:gen_Klembeck} is a consequence of Pinchuk's rescaling method (see~\cite{P1991}):

\begin{theorem}
Suppose $\Omega \subset\Cb^d$ is domain whose boundary is $C^2$ and strongly pseudoconvex in a neighborhood of some point $\xi \in \partial \Omega$. If $x_n \in \Omega$ is a sequence converging to $\xi$, then there exists affine maps $A_n \in \Aff(\Cb^d)$ so that $A_n \Omega$ converges in the local Hausdorff topology to 
\begin{align*}
\Uc = \left\{ (z_1, \dots, z_d) \in \Cb^d: \Imaginary(z_1) > \sum_{i=2}^d \abs{z_i}^2 \right\} 
\end{align*}
and $A_nx_n = (i,0,\dots, 0)$. 
\end{theorem}

See the proof of Theorem 2 in~\cite{G1997} for a detailed argument.

\section{The proof of Theorem~\ref{thm:cont} and Theorem~\ref{thm:upper_cont}}\label{sec:main_result}

\begin{proof}[Proof of Theorem~\ref{thm:cont}]
Suppose for a contradiction that there does not exist such an $\epsilon>0$. Then for each $n > 0$ there exists some bounded convex domain $\Omega_n \subset \Cb^d$ with $C^\infty$ boundary so that 
\begin{align*}
\limsup_{z \rightarrow \partial \Omega_n} \abs{f(\Omega_n, z) - f(\Bc, 0)} < 1/n
\end{align*}
and $\Omega_n$ is not strongly pseudoconvex. Fix some $\xi_n \in \partial \Omega_n$ with $m(\xi_n) = \left(m_1^{(n)}, \dots, m_d^{(n)}\right)$ and $m_d^{(n)} > 2$. 

Now using Theorems~\ref{thm:rescale_finite_type} and~\ref{thm:rescale_inf_type}, for each $n \in \Nb$ we can find a sequence $x_m \rightarrow \xi_n$ and affine maps $A_m$ so that $A_m(\Omega_n, x_m)$ converges in $\Xb_{d,0}$ to $(\wh{\Omega}_n, 0)$ where $\wh{\Omega}_n \in \Lb$. 

We claim that 
\begin{align*}
\abs{f(\wh{\Omega}_n, x) -f(\Bc,0)}< 1/n
\end{align*}
for all $x \in \wh{\Omega}_n$. Now there exists $y_m \in \Omega_n$ so that $A_m y_m \rightarrow x$. Moreover, by Theorem~\ref{thm:dist_conv},
\begin{align*}
\lim_{m \rightarrow \infty} K_{\Omega_n}(x_m, y_m)= \lim_{m \rightarrow \infty} K_{A_m\Omega_n}(A_mx_m, A_my_m) = K_{\wh{\Omega}_n}(x,0).
\end{align*}
Since $K_{\Omega_n}$ is a proper metric and $x_m \rightarrow \xi_n \in \partial \Omega_n$, we then see that 
\begin{align*}
\lim_{m \rightarrow \infty} d_{\Euc}(y_m, \partial \Omega_n) =0.
\end{align*}
Hence 
\begin{align*}
\abs{f(\wh{\Omega}_n, x) -f(\Bc,0)}
&= \limsup_{m \rightarrow \infty} \abs{f(A_m\Omega_n,A_my_m)  -f(\Bc, 0)} \\
& = \limsup_{m \rightarrow \infty} \abs{f(\Omega_n,y_m)  -f(\Bc, 0)} < 1/n.
\end{align*}

Now since $\Lb$ is compact in $\Xb_{d+1}$ we can pass to a subsequence so that $\wh{\Omega}_n$ converges to some $\wh{\Omega}$ in $\Lb$. Now for $x \in \wh{\Omega}$ we have 
\begin{align*}
f(\wh{\Omega}, x) = \lim_{n \rightarrow \infty} f(\wh{\Omega}_n, x) = f(\Bc,0).
\end{align*}
Thus by hypothesis $\wh{\Omega}$ is biholomorphic to $\Bc$. But $\wh{\Omega} \in \Lb$ and so we have a contradiction. 

\end{proof}

\begin{proof}[Proof of Theorem~\ref{thm:upper_cont}] This is essentially identical to the proof of Theorem~\ref{thm:cont}. \end{proof}

\section{The proof of Theorem~\ref{thm:squeezing}}\label{sec:squeezing}

To deduce Theorem~\ref{thm:squeezing} from our general results we only need to show:

\begin{proposition} Suppose $(\Omega_n, x_n)$ is a sequence converging to $(\Omega, x)$ in $\Xb_{d,0}$. Then 
\begin{align*}
\limsup_{n \rightarrow \infty} s_{\Omega_n}(x_n) \leq s_\Omega(x).
\end{align*}
\end{proposition}

\begin{proof} By passing to a subsequence we may assume that 
\begin{align*}
\lim_{n \rightarrow \infty} s_{\Omega_n}(x_n) = \limsup_{n \rightarrow \infty} s_{\Omega_n}(x_n).
\end{align*}
Let $r_n = s_{\Omega_n}(x_n)$ and $r = \lim_{n \rightarrow \infty} s_{\Omega_n}(x_n)$. We may assume that $r > 0$ (otherwise there is nothing to prove).

By Theorem 2.1 in~\cite{DGZ2012}, there exists an injective holomorphic map $f_n: \Omega_n \rightarrow \Bc$  with $f(x_n)=0$ and
\begin{align*}
r_n\Bc \subset f_n(\Omega_n).
\end{align*}
Now by Theorem~\ref{thm:dist_conv}, $K_{\Omega_n}$ converges locally uniformly to $K_\Omega$ and so we can pass to a subsequence so that $f_n$ converges locally uniformly to a holomorphic function $f:\Omega \rightarrow \Bc$. 

Now fix some $w \in r \Bc$. Then for large $n$, we have $f_n^{-1}(w) = \{ z_n\}$ for some $z_n \in \Omega_n$. Then
\begin{align*}
K_{\Omega_n}(z_n, x_n)  = K_{f(\Omega_n)}( f(z_n), 0) \leq K_{ r_n\Bc }(w, 0)
\end{align*}
and so
\begin{align*}
\limsup_{n \rightarrow \infty} K_{\Omega_n}(z_n, x_n)  \leq K_{r\Bc}(w,0).
\end{align*}
So, using the fact that  $K_{\Omega_n}$ converges locally uniformly to $K_\Omega$, we can pass to a subsequence so that $z_n$ converges to some $z \in \Omega$. Then $f(z)=w$. Since $w \in r \Bc$ was arbitrary, we see that $r \Bc \subset f(\Omega)$. 

Finally, by Theorem 2.2 in~\cite{DGZ2012} the map $f:\Omega \rightarrow \Bc$ is injective. Thus 
\begin{align*}
\lim_{n \rightarrow \infty} s_{\Omega_n}(x_n) \leq s_\Omega(x)
\end{align*}
and the proof is complete.
\end{proof}

\section{The proof of Theorem~\ref{thm:gen_riem}}\label{sec:gen_riem}

We begin by showing the following normal family type result:

\begin{lemma}\label{lem:normal_family}
Suppose $\Omega_n$ converges to some $\Omega$ in $\Xb_{d}$ and $g_n \in \Gc_M(\Omega_n)$. Then there exists $n_k \rightarrow \infty$ and a metric $g \in \Gc_M(\Omega)$ so that $g_{n_k}$ converges to $g$ locally uniformly and for all vectors $X,Y,v,w\in \Cb^d$ 
\begin{align*}
X(Y(g_{n_k}(v,w)) \rightarrow X(Y(g(v,w))
\end{align*}
locally uniformly. 
\end{lemma}

\begin{proof} Let $e_1,\dots, e_d$ be the standard basis of $\Cb^d$ and let $u_1=e_1, \dots, u_d=e_d, u_{d+1}=ie_1, \dots, u_{2d} = ie_d$.  Then for $1 \leq i,j,k,\ell \leq 2d$ define the function $f_{n,i,j,k,\ell}:\Omega_n \rightarrow \Cb$ by
\begin{align*}
f_{n,i,j,k,\ell}= u_\ell(u_k(g_n(u_i,u_j)).
\end{align*}

Now if $K \subset \Omega$ is a compact set, there exists $N > 0$ so that $K \subset \Omega_n$ for $n > N$. Using the definition of $\Gc_M(\Omega_n)$, there exists some $C>0$ so that for all $n > N$ and all $1 \leq i,j,k,\ell \leq 2d$, the function $f_{n,i,j,k,\ell}$ is $C$-Lipschitz on $K$ and $\sup_{z \in K}\abs{ f_{n,i,j,k,\ell}(z)} < C$. 

Since $K \subset \Omega$ is an arbitrary compact set, we can pass to a subsequence of the $g_n$ and assume that 
\begin{align*}
\lim_{n \rightarrow \infty} f_{n,i,j,k,\ell}
\end{align*}
exists for all $i,j,k,\ell$ and the convergence is locally uniform. 

Now for $1\leq i,j,k \leq 2d$ define the function $h_{n,i,j,k} : \Omega_n \rightarrow \Cb$ by 
\begin{align*}
h_{n,i,j,k}= u_k(g(u_i, u_j)).
\end{align*}
Fix some $z_0 \in \Omega$. Then there exists some $N >0$ so that $z_0 \in \Omega_n$ for all $n > N$. Using part (3) of the definition of $\Gc_M$, we see that $\sup_{n > N} h_{n,i,j}(z_0)$ is finite. Then we can pass to a  subsequence and assume that 
\begin{align*}
\lim_{n \rightarrow \infty} h_{n,i,j}(z_0)
\end{align*}
exists for all $1 \leq i,j \leq d$. Then using the fact that $u_\ell(h_{n,i,j}) = f_{n,i,j,k,\ell}$, we can pass to a subsequence so that 
\begin{align*}
\lim_{n \rightarrow \infty} h_{n,i,j}
\end{align*}
exists and the convergence is locally uniformly in the $C^1$ topology. 

Repeating the argument above, we can then pass to a subsequence of the $g_n$ converges locally uniformly to some $C^2$ symmetric 2-form $g$ and for all vectors $X,Y,v,w \in \Cb^d$ 
\begin{align*}
X(Y(g_{n_k}(v,w)) \rightarrow X(Y(g(v,w))
\end{align*}
locally uniformly. Part (2) of the definition of $\Gc_M$ implies that $g$ is an actual metric and hence is in $\Gc_M(\Omega)$.
\end{proof}

Now for $M,d >0$, define a function $h_M:\Xb_{d,0} \rightarrow \Rb$ by letting $h_{M}(\Omega, x)$ be the infinium of all numbers $\epsilon > 0$ so that there exists a metric $g \in \Gc_M(\Omega)$ with
\begin{align*}
\max_{ v,w \in T_z \Omega \setminus \{0\} } \abs{ H_g(v)- H_g(w) } \leq \epsilon \text{ for all } z \in B_\Omega(x;1/\epsilon).
\end{align*}
Here $B_\Omega(x;r)$ is the closed ball of radius $r$ about the point $x \in \Omega$ with respect to the Kobayashi distance. 

\begin{proposition}\label{prop:normal}
The function $(-h_{M}) : \Xb_{d,0} \rightarrow \Rb_{\leq 0}$ is an upper semicontinuous intrinsic function. 
\end{proposition}

\begin{proof} We first argue that $h_M$ is intrinsic. Let $\varphi:\Omega_1 \rightarrow \Omega_2$ be a biholomorphism with $\varphi(x_1)=x_2$. Suppose that $g \in \Gc_M(\Omega_1)$. Then consider the metric $\varphi^*(g)$ on $\Omega_2$ given by 
\begin{align*}
\varphi^*(g)_z(v,w) = g_{\varphi^{-1}}\left( d(\varphi^{-1})_z(v), d(\varphi^{-1})_z(w)\right).
\end{align*}
Since $\varphi$ is a biholomorphism the metric $\varphi^*(g)$ is K{\"a}hler and since $\varphi$ is an isometry with respect to the Kobayashi metric and distance, it is easy to show that $\varphi^*(g) \in \Gc_M(\Omega_2)$. Finally, we observe that $H_g(v) = H_{\varphi^*(g)}(d(\varphi)(v))$ and $\varphi(B_{\Omega_1}(x;r)) = B_{\Omega_2}(\varphi(x);r)$. So we see that $h_M(\Omega_1, x_1) \geq h_M(\Omega_2,x_2)$. Repeating the above argument with $\varphi^{-1}:\Omega_2 \rightarrow \Omega_2$ shows that $h_M(\Omega_1, x_1) \leq h_M(\Omega_2,x_2)$. Hence $h_M$ is intrinsic. 

The fact that $(-h_M)$ is upper semicontinuous follows immediately from Lemma~\ref{lem:normal_family} and Theorems~\ref{thm:dist_conv} and~\ref{thm:metric_conv}. \end{proof}

\begin{proposition}\label{prop:rigidity}
Suppose $\Omega \in \Xb_d$, $M > 1$, and $h_M(\Omega,x) = 0$ for some $x \in \Omega$. Then $\Omega$ is biholomorphic to $\Bc$. 
\end{proposition}

\begin{proof}
Using Lemma~\ref{lem:normal_family}, there exists a metric $g \in \Gc_M(\Omega)$ with holomorphic sectional curvature which only depends on the base point. Then Theorem 7.5 in~\cite[Chapter IX, Section 7]{KN1996} implies that $g$ has constant holomorphic sectional curvature. Thus by Theorems 7.8 and 7.9 in~\cite[Chapter IX, Section 7]{KN1996}, $\Omega$ is biholomorphic to the complex projective space, $\Bc$, or $\Cb^d$. Since $\Omega$ is non-compact, clearly $\Omega$ is not biholomorphic to the complex projective space. Since the Kobayashi metric is nondegenerate on $\Omega$ (see Remark~\ref{rmk:barth}), $\Omega$ cannot be biholomorphic to $\Cb^d$. Thus $\Omega$ is biholomorphic to $\Bc$.    
\end{proof}

Combining Proposition~\ref{prop:normal}, Proposition~\ref{prop:rigidity}, and Theorem~\ref{thm:upper_cont} we obtain:

\begin{corollary}
For any $d > 0$ and $M > 1$, there exists some $\epsilon = \epsilon(M,d) > 0$ so that: if $\Omega \subset \Cb^d$ is a bounded convex domain with $C^\infty$ boundary and  
\begin{align*}
h(\Omega,z) \leq \epsilon
\end{align*}
outside a compact subset of $\Omega$, then $\Omega$ is strongly pseudoconvex.
\end{corollary}

Finally we are ready to prove Theorem~\ref{thm:gen_riem}.

\begin{proof}[Proof of Theorem~\ref{thm:gen_riem}]
Fix $\epsilon > 0$ with the the following property: if $\Omega \subset \Cb^d$ is a bounded convex domain with $C^\infty$ boundary and  
\begin{align*}
h(\Omega,z) \leq 2\epsilon
\end{align*}
outside a compact set of $\Omega$, then $\Omega$ is strongly pseudoconvex.

Now suppose that $\Omega \in \Xb_d$, $K \subset \Omega$ is compact, and there exists a metric $g \in \Gc_M(\Omega)$ so that 
\begin{align*}
\max_{ v \in T_z \Omega \setminus \{0\} } \abs{ H_{g}(v)- \frac{-4}{d+1} } \leq \epsilon \text{ for all } z \in \Omega \setminus K.
\end{align*}
We claim that $\Omega$ is strongly pseudoconvex. 

Since $K_\Omega$ is a proper distance on $\Omega$ (see Remark~\ref{rmk:barth}), there exists some compact subset $K^\prime \subset \Omega$ so that $B_\Omega(x;1/\epsilon) \subset \Omega \setminus K$ for all $x \in \Omega \setminus K^\prime$. Then, with this choice of $K^\prime$, 
\begin{align*}
h_M(\Omega,x) \leq 2\epsilon
\end{align*}
for all $x \in \Omega \setminus K^\prime$. So by our choice of $\epsilon > 0$, $\Omega$ is strongly pseudoconvex. 
\end{proof}

\section{Proof of Theorem~\ref{thm:bergman}}\label{sec:bergman}

To deduce Theorem~\ref{thm:bergman} from Theorem~\ref{thm:gen_riem} it is enough to prove the following: 

\begin{proposition}\label{prop:bergman}
For any $d > 0$, there exists $M=M(d)>1$ so that:  $b_\Omega \in \Gc_M(\Omega)$ for every $\Omega \in \Xb_d$.
\end{proposition}

\begin{proof} By a result of Frankel~\cite{F1989b}, there exists some $C_0=C_0(d) > 1$ so that 
\begin{align}
\label{est:one}
\frac{1}{C_0} b_\Omega \leq k_\Omega \leq C_0 b_\Omega
\end{align}
for all $\Omega \in \Xb_d$. So we only need to prove conditions (3), (4), and (5) in the definition $\Gc_M(\Omega)$. We will actually prove a stronger assertion: for any $n > 0$, there exists a $C_n =C_n(d) > 0$ so that: for all vectors $X_1, \dots, X_n, v,w \in \Cb^d$ we have
\begin{align*}
X_1 \cdots X_n (b_{\Omega,z}(v,w)) \leq C_n k_\Omega(z;v) k_\Omega(z;w)\prod_{i=1}^n k_\Omega(z;X_i) .
\end{align*}

Suppose not, then there exists some $n >0$ and sequences $(\Omega_m,x_m) \in \Xb_{d,0}$, $X_{m,1}, \dots, X_{m,n}, v_m, w_m \in \Cb^d$
so that 
\begin{align*}
\lim_{m \rightarrow \infty} \frac{ X_{m,1} \cdots X_{m,n} (b_{\Omega_m,x_m}(v_m,w_m))}{k_{\Omega_m}(x_m;v_m) k_{\Omega_m}(x_m;w_m)\prod_{i=1}^n k_{\Omega_m}(x_m;X_{m,i}) } = \infty.
\end{align*}

Now, for any affine isomorphism $A \in \Aff(\Cb^d)$ and vectors $X_1, \dots, X_n, v,w \in \Cb^d$ we have 
\begin{align*}
X_1 \cdots X_n (b_{\Omega}(v,w)) = (AX_1) \cdots (AX_n) (b_{A\Omega}(Av,Aw))
\end{align*}
So using the invariance of $k_{\Omega_n}$ we can replace each tuple 
\begin{align*}
(\Omega_m, x_m), X_{m,1}, \dots, X_{m,n}, v_m, w_m
\end{align*}
 by 
 \begin{align*}
(A_m\Omega_m, A_mx_m), A_mX_{m,1}, \dots, A_mX_{m,n}, A_mv_m, A_mw_m
 \end{align*}
 where $A_m$ is some affine isomorphism. So using Theorem~\ref{thm:frankel_compactness}, we can suppose that $(\Omega_m, x_m)$ converges to some $(\Omega, x)$ in $\Xb_{d,0}$.  
We also note that the ratio
\begin{align*}
\frac{ X_{m,1} \cdots X_{m,n} (b_{\Omega_m,x_m}(v_m,w_m))}{k_{\Omega_m}(x_m;v_m) k_{\Omega_m}(x_m;w_m)\prod_{i=1}^n k_{\Omega_m}(x_m;X_{m,i}) }
\end{align*}
is invariant under scaling any of the $X_{m,1}, \dots, X_{m,n}, v_m,w_m$ so we may assume that these are all unit vectors with respect to the Euclidean metric. Then, after passing to a  subsequence, we can suppose that 
\begin{align*}
X_{m,1}, \dots, X_{m,n}, v_m, w_m \rightarrow X_1, \dots, X_n, v,w.
\end{align*}

Now by Theorem~\ref{thm:metric_conv}
\begin{align*}
\lim_{m \rightarrow \infty} & \ k_{\Omega_m}(x_m;v_m) k_{\Omega_m}(x_m;w_m)\prod_{i=1}^n k_{\Omega_m}(x_m;X_{m,i})  \\
& = k_{\Omega}(x;v) k_{\Omega}(x;w)\prod_{i=1}^n k_{\Omega}(x;X_{i}) > 0.
\end{align*} 
So to obtain a contradiction we must show that 
\begin{align*}
\limsup_{m \rightarrow \infty} \abs{X_{m,1} \cdots X_{m,n} (b_{\Omega_m,x_m}(v_m,w_m))}< +\infty.
\end{align*}

Let $\kappa_{\Oc} : \Oc \times \Oc \rightarrow \Cb$ be the Bergman kernel on a domain $\Oc$. Then (by definition)
\begin{align*}
b_{\Oc,z}(v,w) = \sum_{i=1}^d v_i \left(\frac{\partial^2}{\partial z_i \partial \overline{z}_i} \log \kappa_{\Oc}(z,z) \right)\overline{w}_i.
\end{align*}
We also have the following basic properties:
\begin{enumerate}
\item If $\Oc^\prime \subset \Oc$, then $\kappa_{\Oc}(z,z) \leq \kappa_{\Oc^\prime}(z,z)$.
\item If $z,w \in \Oc$, then $\kappa_{\Oc}(z,w) \leq \kappa_{\Oc}(z,z)\kappa_{\Oc}(w,w)$.
\end{enumerate}

We claim that after passing to a subsequence $\kappa_{\Omega_m}$ converges to some $C^\infty$ function $\kappa : \Omega \times \Omega \rightarrow \Cb$ and the convergence is locally uniform for each partial derivative. Using basic properties of harmonic functions, it is enough to show that $\kappa_{\Omega_m}$ converges locally uniformly to some function $\kappa : \Omega \times \Omega \rightarrow \Cb$. Using Montel's theorem, it is enough to show that $\kappa_{\Omega_m}$ is locally bounded. 

So fix a compact set $K \subset \Omega$ and open neighborhood $\Oc$ of $K$ which is relatively compact in $\Omega$. Then for $m$ large, $\Oc \subset \Omega_m$ and so 
\begin{align*}
\kappa_{\Omega_m}(z,w) \leq \kappa_{\Omega_m}(z,z)\kappa_{\Omega_m}(w,w) \leq \kappa_{\Oc}(z,z)\kappa_{\Oc}(w,w)
\end{align*}
for $z,w \in K$. So $\kappa_{\Omega_m}$ is locally bounded. 

Then by Montel's theorem we can pass to a subsequence and assume that $\kappa_{\Omega_m}$ converges to some $C^\infty$ function $\kappa : \Omega \times \Omega \rightarrow \Cb$ and the convergence is locally uniform for each partial derivative. Now using Estimate~\ref{est:one}, we see that $1/\kappa_{\Omega_m}(z,z)$ is locally bounded and thus 
\begin{align*}
\limsup_{m \rightarrow \infty} \abs{X_{m,1} \cdots X_{m,n} (b_{\Omega_m,x_m}(v_m,w_m))}< +\infty.
\end{align*}

So we have a contradiction. 
\end{proof}

\appendix

\bibliographystyle{alpha}
\bibliography{complex_kob}

\begin{thebibliography}{DFW14}

\bibitem[Bar80]{B1980}
Theodore~J. Barth.
\newblock Convex domains and {K}obayashi hyperbolicity.
\newblock {\em Proc. Amer. Math. Soc.}, 79(4):556--558, 1980.

\bibitem[CP01]{CP2001}
B.~Coupet and S.~Pinchuk.
\newblock Holomorphic equivalence problem for weighted homogeneous rigid
  domains in {${\Bbb C}^{n+1}$}.
\newblock In {\em Complex analysis in modern mathematics ({R}ussian)}, pages
  57--70. FAZIS, Moscow, 2001.

\bibitem[DFW14]{DFW2014}
K.~Diederich, J.~E. Forn{\ae}ss, and E.~F. Wold.
\newblock Exposing points on the boundary of a strictly pseudoconvex or a
  locally convexifiable domain of finite 1-type.
\newblock {\em J. Geom. Anal.}, 24(4):2124--2134, 2014.

\bibitem[DGZ12]{DGZ2012}
Fusheng Deng, Qian Guan, and Liyou Zhang.
\newblock Some properties of squeezing functions on bounded domains.
\newblock {\em Pacific J. Math.}, 257(2):319--341, 2012.

\bibitem[DGZ16]{DGZ2016}
Fusheng Deng, Qi'an Guan, and Liyou Zhang.
\newblock Properties of squeezing functions and global transformations of
  bounded domains.
\newblock {\em Trans. Amer. Math. Soc.}, 368(4):2679--2696, 2016.

\bibitem[Fra89]{F1989b}
Sidney Frankel.
\newblock Affine approach to complex geometry.
\newblock In {\em Recent developments in geometry ({L}os {A}ngeles, {CA},
  1987)}, volume 101 of {\em Contemp. Math.}, pages 263--286. Amer. Math. Soc.,
  Providence, RI, 1989.

\bibitem[Gau97]{G1997}
Herv{\'e} Gaussier.
\newblock Characterization of convex domains with noncompact automorphism
  group.
\newblock {\em Michigan Math. J.}, 44(2):375--388, 1997.

\bibitem[Haw53]{H1953}
N.~S. Hawley.
\newblock Constant holomorphic curvature.
\newblock {\em Canadian J. Math.}, 5:53--56, 1953.

\bibitem[Igu54]{I1954}
Jun-ichi Igusa.
\newblock On the structure of a certain class of {K}aehler varieties.
\newblock {\em Amer. J. Math.}, 76:669--678, 1954.

\bibitem[Kle78]{K1978}
Paul~F. Klembeck.
\newblock K\"ahler metrics of negative curvature, the {B}ergmann metric near
  the boundary, and the {K}obayashi metric on smooth bounded strictly
  pseudoconvex sets.
\newblock {\em Indiana Univ. Math. J.}, 27(2):275--282, 1978.

\bibitem[KN96]{KN1996}
Shoshichi Kobayashi and Katsumi Nomizu.
\newblock {\em Foundations of differential geometry. {V}ol. {II}}.
\newblock Wiley Classics Library. John Wiley \& Sons, Inc., New York, 1996.
\newblock Reprint of the 1969 original, A Wiley-Interscience Publication.

\bibitem[LSY04]{LSY2004}
Kefeng Liu, Xiaofeng Sun, and Shing-Tung Yau.
\newblock Canonical metrics on the moduli space of {R}iemann surfaces. {I}.
\newblock {\em J. Differential Geom.}, 68(3):571--637, 2004.

\bibitem[Pin91]{P1991}
Sergey Pinchuk.
\newblock The scaling method and holomorphic mappings.
\newblock In {\em Several complex variables and complex geometry, {P}art 1
  ({S}anta {C}ruz, {CA}, 1989)}, volume~52 of {\em Proc. Sympos. Pure Math.},
  pages 151--161. Amer. Math. Soc., Providence, RI, 1991.

\bibitem[Yeu09]{Y2009}
Sai-Kee Yeung.
\newblock Geometry of domains with the uniform squeezing property.
\newblock {\em Adv. Math.}, 221(2):547--569, 2009.

\bibitem[Yu92]{Y1992}
Ji~Ye Yu.
\newblock Multitypes of convex domains.
\newblock {\em Indiana Univ. Math. J.}, 41(3):837--849, 1992.

\bibitem[Zim16]{Z2016}
Andrew~M. Zimmer.
\newblock Gromov hyperbolicity and the {K}obayashi metric on convex domains of
  finite type.
\newblock {\em Math. Ann.}, 365(3-4):1425--1498, 2016.

\end{thebibliography}

\end{document}